\documentclass[11pt]{amsart}
\usepackage{amssymb,enumerate}
\newtheorem{theorem}{Theorem}
\newtheorem{lemma}[theorem]{Lemma}
\newtheorem{proposition}[theorem]{Proposition}
\theoremstyle{remark}
\newtheorem*{remarks}{Remarks}

\numberwithin{theorem}{section}
\numberwithin{equation}{section}
\newcommand{\R}{\mathbb{R}}
\newcommand{\C}{\mathbb{C}}
\newcommand{\Q}{\mathbb{Q}}
\newcommand{\Z}{\mathbb{Z}}
\DeclareMathOperator{\ord}{ord}
\DeclareMathOperator{\proj}{proj}
\DeclareMathOperator{\textspan}{span}
\DeclareMathOperator{\GL}{GL}
\DeclareMathOperator{\tr}{tr}
\DeclareMathOperator{\cond}{cond}
\DeclareMathOperator{\Mat}{Mat}

\begin{document}
\title{Zeros of $L$-functions outside the critical strip}

\author{Andrew R.~Booker}
\address{School of Mathematics, University of Bristol,
University Walk, Bristol, BS8 1TW, United Kingdom}
\email{andrew.booker@bristol.ac.uk}
\thanks{A.~R.~B.\ was supported by EPSRC Grants EP/H005188/1,
EP/L001454/1 and EP/K034383/1.}
\thanks{F.~T.'s work was 
partially supported by the National Science Foundation under Grant No. DMS-1201330.}

\author{Frank Thorne}
\address{Department of Mathematics, University of South Carolina,
1523 Greene Street, Columbia, SC 29208, USA}
\email{thorne@math.sc.edu}

\begin{abstract}
For a wide class of Dirichlet series associated to automorphic forms,
we show that those without Euler products must have zeros
within the region of absolute convergence.  For instance, we prove that
if $f\in S_k(\Gamma_1(N))$ is a classical holomorphic modular form whose
$L$-function does not vanish for $\Re(s)>\frac{k+1}2$, then $f$ is a Hecke
eigenform.  Our proof adapts and extends work of Saias and Weingartner
\cite{SW}, who proved a similar result for degree $1$ $L$-functions.
\end{abstract}

\maketitle
\section{Introduction}
In \cite{SW}, Saias and Weingartner showed that if
$L(s)=\sum_{m=1}^\infty\frac{\lambda(m)}{m^s}$ is a Dirichlet series
with periodic coefficients, then either $L(s)=0$ for some $s$ with real
part $>1$, or $\lambda(m)$ is multiplicative at almost all primes (so
that $L(s)=D(s)L(s,\chi)$ for some primitive Dirichlet character $\chi$
and finite Dirichlet series $D$).  Earlier work of Davenport and Heilbronn
\cite{dh1,dh2} established this result for the special case of the Hurwitz
zeta-function $\zeta(s,\alpha)$ with rational parameter $\alpha$, and
proved an analogue for the degree $2$ Epstein zeta-functions. Also in
degree $2$, Conrey and Ghosh \cite{cg} showed that the $L$-function
associated to the square of Ramanujan's $\Delta$ modular form has
infinitely many zeros outside of its critical strip.  In this paper, we
generalize all of these results and study the extent to which, among all
Dirichlet series associated to automorphic forms (appropriately defined),
the existence of an Euler product is characterized by non-vanishing in
the region of absolute convergence.  For instance, for classical degree
$2$ $L$-functions, we prove the following:
\begin{theorem}\label{thm:classical}
Let $f\in S_k(\Gamma_1(N))$ be a holomorphic cuspform of arbitrary
weight and level. If the associated complete $L$-function
$\Lambda_f(s)=\int_0^\infty f(iy)y^{s-1}\,dy$ does not vanish for
$\Re(s)>\frac{k+1}2$ then $f$ is an eigenfunction of the Hecke operators
$T_p$ for all primes $p\nmid N$.
\end{theorem}

Our method is sufficiently general to apply to $L$-functions of all
degrees, and in fact we obtain Theorem~\ref{thm:classical} as a corollary
of the following general result:
\begin{theorem}\label{thm:main}
Fix a positive integer $n$. For $j=1,\ldots,n$, let $r_j$ be
a positive integer and $\pi_j$ a unitary cuspidal automorphic
representation of $\GL_{r_j}(\mathbb{A}_\Q)$ with $L$-series
$L(s,\pi_j)=\sum_{m=1}^\infty\lambda_j(m)m^{-s}$.  Assume that the
$\pi_j$ satisfy the generalized Ramanujan conjecture at all finite places
(so that, in particular, $|\lambda_j(p)|\le r_j$ for all primes $p$)
and are pairwise non-isomorphic. Let
$$
R=\left\{\sum_{m=1}^M\frac{a_m}{m^s}:M\in\Z_{\ge0},
(a_1,\ldots,a_M)\in\C^M\right\}
$$
denote the ring of finite Dirichlet series, and let $P\in
R[x_1,\ldots,x_n]$ be a polynomial with coefficients in $R$.
Then either $P(L(s,\pi_1),\ldots,L(s,\pi_n))$ has a zero with real
part $>1$ or $P=D(s)x_1^{d_1}\cdots x_n^{d_n}$ for some $D\in R$,
$d_1,\ldots,d_n\in\Z_{\ge0}$.
\end{theorem}

\begin{remarks}\hspace{1cm}
\begin{enumerate}
\item For $\pi_j$ as in the statement of the theorem, it is known
(see \cite{jacquet-shalika}) that $L(s,\pi_j)$ does not vanish for
$\Re(s)\ge1$. Thus if $P=D(s)x_1^{d_1}\cdots x_n^{d_n}$ is a monomial
then whether or not $P(L(s,\pi_1),\ldots,L(s,\pi_n))$ vanishes
for some $s$ with $\Re(s)>1$ is determined entirely by the finite
Dirichlet series $D(s)$.  Further, the Grand Riemann Hypothesis (GRH)
predicts that each $L(s,\pi_j)$ does not vanish for $\Re(s)>\frac12$.
Theorem~\ref{thm:main} demonstrates that the GRH, if it is true,
is a very rigid phenomenon.

\item By the almost-periodicity of Dirichlet series, if
$P(L(s,\pi_1),\ldots,L(s,\pi_n))$ has at least one zero with real part $>1$
then it must have infinitely many such zeros. In fact, our proof shows
that there is some number $\eta=\eta(P;\pi_1,\ldots,\pi_n)>0$ such that
for any $\sigma_1,\sigma_2$ with $1<\sigma_1<\sigma_2\le1+\eta$, we have
\begin{equation}\label{eqn:Nlowerbound}
\#\bigl\{s\in\C:\Re(s)\in[\sigma_1,\sigma_2], \Im(s)\in[-T,T],
P(L(s,\pi_1),\ldots,L(s,\pi_n))=0\bigr\}
\gg T
\end{equation}
for $T$ sufficiently large (where both the implied constant and the
meaning of ``sufficiently large'' depend on $\sigma_1,\sigma_2$
as well as $P$ and $\pi_1,\ldots,\pi_n$).

On the other hand, if we restrict to $\C$-linear combinations
(i.e.\ homogeneous degree $1$ polynomials $P\in\C[x_1,\ldots,x_n]$)
and $\pi_1,\ldots,\pi_n$ with a common conductor and archimedean
component $\pi_{1,\infty}\cong\ldots\cong\pi_{n,\infty}$, Bombieri and
Hejhal \cite{bh} showed, subject to GRH and a weak form of the pair
correlation conjecture for $L(s,\pi_j)$, that asymptotically $100\%$
of the non-trivial zeros of $P(L(s,\pi_1),\ldots,L(s,\pi_n))$ have real
part $\frac12$.

\item The assumption of the Ramanujan conjecture in Theorem~\ref{thm:main}
could be relaxed. For instance, it would suffice to have, for each
fixed $j$:
\begin{itemize}
\item[(i)]some mild control over the
coefficients of the logarithmic derivative
$$\frac{L'}{L}(s,\pi_j)=\sum_{m=1}^\infty c_j(m)m^{-s}$$
at prime powers, namely
$\sum_p\frac{|c_j(p^k)|^2}{p^k}<\infty$
for any fixed $k\ge2$ (cf.\ \cite[Hypothesis H]{rs});
\item[(ii)]an average bound for $|\lambda_j(p)|^4$ over arithmetic
progressions of primes, namely
$$
\limsup_{x\to\infty}
\frac{\sum_{\substack{p\le x\\p\equiv a\;(\text{mod }q)}}|\lambda_j(p)|^4}
{\sum_{\substack{p\le x\\p\equiv a\;(\text{mod }q)}}1}
\le C_j,
$$
for all co-prime $a,q\in\Z_{>0}$, where $C_j>0$ is independent of $a,q$.
\end{itemize}
Note that (i) is known to hold when $r_j\le 4$ (see \cite{rs,kim}).
Further, both estimates follow from the Rankin--Selberg method if,
for instance, the tensor square $\pi_j\otimes\pi_j$ is automorphic
for each $j$. Since this is known when $r_j=2$ (see \cite{gj}),
Theorem~\ref{thm:main} could be extended to include the $L$-functions
associated to Maass forms.

\item The main tool used in the proof is the quasi-orthogonality of
the coefficients $\lambda_j(p)$, i.e.\ asymptotic estimates for sums of
the form $\sum_{p\le x}\frac{\lambda_j(p)\overline{\lambda_k(p)}}p$
as $x\to\infty$. These follow from the Rankin--Selberg method, and
were obtained in a precise form independently by
Wu--Ye \cite[Thm.~3]{wu-ye} and Avdispahi\'c--Smajlovi\'c \cite[Thm.~2.2]{as}.
(We also make use of similar estimates for sums over $p$ in an arithmetic
progression---see Lemma~\ref{lem:sump} for the exact statement---though
it is likely that this could be avoided at the expense of making the
proof more complicated.)

Since quasi-orthogonality and the Ramanujan conjecture are essentially
the only properties of automorphic $L$-functions that we require,
one could instead take these as hypotheses and state the theorem for
an axiomatically-defined class of $L$-functions, such as the Selberg
class. However, it has been conjectured that the Selberg
class coincides with the class of automorphic $L$-functions, so this
likely offers no greater generality.

\item The conclusion of Theorem~\ref{thm:main} is interesting
even for $n=1$. For instance, Nakamura and Pa\'nkowski \cite {np}
have shown very recently, for a wide class of $L$-functions $L(s)$,
that if $P\in R[x]$ is not a monomial and $\delta>0$ then $P(L(s))$
necessarily has zeros in the half-plane $\Re(s)>1-\delta$. Our result
strengthens this to $\Re(s)>1$. (On the other hand,
\cite{np} also yields the estimate \eqref{eqn:Nlowerbound} for any
$[\sigma_1,\sigma_2]\subseteq(\frac12,1)$, which does not follow from
our method.)

\item Our results are related to universality results 
for zeta and $L$-functions. Voronin \cite{voronin} proved for any compact
set $K$ with connected complement contained within the strip $\Re(s) \in (\frac12, 1)$, and any nonvanishing, continuous 
function $f :  K \rightarrow \C$ holomorphic on the interior of $K$, that $f$ can be uniformly
approximated by vertical translates of the zeta function. 

Voronin's results were extended
by a number of authors. One result similar to ours, due to
Laurin\v{c}ikas and Matsumoto \cite{LM}, states that given $m$ functions $f_1, \dots, f_m$ as above, and
$L$-functions $L_j(s, F)$ associated to twists of a Hecke newform $F$ by pairwise inequivalent Dirichlet characters, that the $f_j$ may
be simultaneously approximated by a single vertical translate of the functions $L_j(s, F)$.
This implies \cite[Theorem 4]{LM} that non-trivial linear combinations of the $L_j(s, F)$ must
contain zeros inside the critical strip with $\Re(s) > \frac{1}{2}$.

References to many more works on universality can be found in \cite{LM}.

\end{enumerate}
\end{remarks}

\subsection*{Summary of the proof} Our proof closely follows Saias and Weingartner's in broad outline, but becomes more technical in
some places. The reader may wish to read \cite{SW} first.

The technical heart of our paper is Proposition \ref{prop:lem2}, an extension of Lemma 2 of \cite{SW}. 
Given $n$ complex numbers $z_1, \cdots, z_n$ (bounded away from $0$ and $\infty$), we would like to 
simultaneously solve the equations $L(s, \pi_j) = z_j$, leading to a quick proof of the main theorem.
As a substitute, we solve equations of a form $\prod_{p > y} L(\sigma + it_p, \pi_{j, p}) = z_j$, where the ordinate of $s$ is allowed to vary for each
prime.

Given this, in Section \ref{sec:main} we prove our 
main theorem, following the proof of Theorem 2 in \cite{SW}.
As in \cite{SW}, the main tools are Weyl's criterion, allowing us to simultaneously approximate all of the $p^{-\sigma - it_p}$ by
$p^{- \sigma - it}$ for a single $t$, and Rouch\'e's theorem, which states that actual zeros must exist near approximate zeros.

The proof of Proposition \ref{prop:lem2} follows those of Lemmas 1 and 2 of \cite{SW}. However, in \cite{SW} the Dirichlet
coefficients $\lambda(m)$ are all periodic to some fixed modulus,
and this fact, combined with the prime number theorem for arithmetic progressions, allows for easy control of various partial sums that need
to be estimated.
Here, we must do without this periodicity.

To prove Proposition \ref{prop:lem2}, we choose (in Proposition \ref{prop:sw1}) a partition of the set of primes $p > y$ into disjoint subsets $S$, and  complex numbers
$\epsilon_p \in S^1$ for each $p > y$, so that the vectors of partial sums 
$\sum_{p \in S} \epsilon_p \lambda_j(p) p^{-\sigma}$ are linearly independent in a precise quantitative
sense. Our main tool is the Rankin--Selberg method (substituting for
periodicity and orthogonality of Dirichlet characters);
see Lemma~\ref{lem:sump}. 

We also rely on the rather technical Proposition \ref{prop:td}, which says that
for matrices $g_1, \dots, g_m$, 
we can continuously solve equations of the form
$\sum_{i = 1}^m g_i f_i(z) = z$ for $n$-tuples of complex numbers
$z = (z_1, \cdots, z_n)$. The $g_i$ are constructed from the sums over $p \in S$ considered in Proposition 
\ref{prop:sw1}, but we are able to formulate Proposition \ref{prop:td} in a general manner, without reference to automorphic
forms or primes.

The conclusion of Proposition \ref{prop:td} is guaranteed only for large $m$, so that the number of subsets $S$ needed may be large. We choose these subsets to be arithmetic
progressions, for which the Rankin--Selberg 
estimates presented in Lemma \ref{lem:sump} are known to hold. If such estimates were unavailable, it seems likely that
we could still obtain our result by constructing the $S$ in a more {\itshape ad hoc} fashion
instead. In any case, and in contrast to Saias--Weingartner, the modulus of the arithmetic progression has no particular arithmetic significance,
and is chosen to be coprime to all the conductors of the $\pi_j$.

\subsection*{Acknowledgements}
This work was carried out during visits by both authors to the Research
Institute for Mathematical Sciences and Kyoto University. We thank these
institutions and our hosts, Professors Akio Tamagawa and Akihiko Yukie,
for their generous hospitality. We also thank the referee for helpful comments.

\section{Preliminaries}
\subsection{Automorphic $L$-functions}
Let $\pi_j$ be as in the statement of Theorem~\ref{thm:main}. Each
$\pi_j$ can be written as a restricted tensor product
$\pi_{j,\infty}\otimes\bigotimes_p\pi_{j,p}$ of local representations,
where $p$ runs through all prime numbers.  Then we have
\begin{equation}\label{eq:l_def_1}
L(s,\pi_j)=\prod_pL(s,\pi_{j,p}),\quad\text{for }\Re(s)>1.
\end{equation}
Here each local factor $L(s,\pi_{j,p})$ is a rational function of
$p^{-s}$, of the form
\begin{equation}\label{eq:l_def_2}
L(s,\pi_{j,p})=\frac1{(1-\alpha_{j,p,1}p^{-s})\cdots(1-\alpha_{j,p,r_j}p^{-s})}
\end{equation}
for certain complex numbers $\alpha_{j,p,\ell}$. The generalized Ramanujan
conjecture asserts that
$|\alpha_{j,p,\ell}|\le 1$, with equality holding for all
$p\nmid\cond(\pi_j)$, where $\cond(\pi_j)\in\Z_{>0}$ is the conductor of
$\pi_j$. In particular,
$|\lambda_j(p)|=|\alpha_{j,p,1}+\ldots+\alpha_{j,p,r_j}|\le r_j$.

\begin{lemma}\label{lem:sump}\footnote{[Added after publication.] As Mattia Righetti pointed out to us, this lemma is incorrect
as claimed, although the statement is true with $O(\sigma - 1)$ replaced by $O\big( \frac{1}{\log(2/(\sigma - 1))}\big)$. The final
line of the proof does not follow with the uniformity claimed, but in a published correction we prove that the bound
$\sum_{p > y} p^{-\sigma} \gg \frac{y^{1 - \sigma}}{\log y} \log \frac{2}{\sigma - 1}$ does.

The lemma is used in the proof of Proposition \ref{prop:sw1}, where it is needed only that the error tend to $0$
as $\sigma \rightarrow 1^+$. The corrected error term is sufficient, and the remaining results remain valid with only
cosmetic changes to the proofs.}
Let $a$ and $q$ be positive integers satisfying
$\bigl(q,a\prod_{j=1}^n\cond(\pi_j)\bigr)=1$.
Then
\[
\sum_{\substack{p>y\\p\equiv{a}\;(\text{\rm mod }q)}}
\frac{|u_1\lambda_1(p)+\ldots+u_n \lambda_n(p)|^2}{p^\sigma} = 
\bigg( \frac{1}{ \phi(q)} + O(\sigma-1)\bigg)
\sum_{p>y}p^{-\sigma}
\]
for all $y>0$, $\sigma\in(1,2]$ and all unit vectors
$(u_1, \ldots, u_n)$, where the implied constant
depends only on $\pi_1,\ldots,\pi_n$ and $q$.
\end{lemma}
\begin{proof}
Let $\chi\;(\text{mod }q)$ be a Dirichlet character, not necessarily primitive.
We consider the sum
\[
E_{jk\chi}(x)=
\sum_{p\le x}\big(\lambda_j(p)\overline{\lambda_k(p)}\chi(p)
-\delta_{jk\chi}\big)\frac{\log{p}}{p},
\]
running over primes $p\le x$,
where $\delta_{jk\chi}=1$ if $j=k$ and $\chi$ is the trivial
character, and $0$ otherwise.  Applying \cite[(2) and (3)]{as} with
$(\pi,\pi')=(\pi_j\otimes\chi,\pi_k)$ and, if $\chi$ is imprimitive,
subtracting any contribution from the terms with $p|q$,
we obtain the bound $E_{jk\chi}(x)\ll_q 1$.

Next, for any non-integral $y\ge\frac32$ and any $\sigma\in(1,2]$, we have
$$
\sum_{p>y}
\frac{\lambda_j(p)\overline{\lambda_k(p)}\chi(p)-\delta_{jk\chi}}{p^\sigma}
=\int_y^\infty\frac{t^{1-\sigma}}{\log{t}}\,dE_{jk\chi}(t).
$$
Integrating by parts and applying the above estimate for
$E_{jk\chi}$, we see that this is $\ll_q y^{1-\sigma}/\log{y}$.

Now, expanding the square and using orthogonality of Dirichlet
characters, we have
$$
\begin{aligned}
\sum_{\substack{p>y\\p\equiv{a}\;(\text{mod }q)}}
\frac{|u_1\lambda_1(p)+\ldots+u_n\lambda_n(p)|^2}{p^\sigma}&=
\frac1{\phi(q)}\sum_{j=1}^n\sum_{k=1}^n
\sum_{\chi\;(\text{mod }q)}u_j\overline{u_k}\overline{\chi(a)}
\sum_{p>y}\frac{\lambda_j(p)\overline{\lambda_k(p)}\chi(p)}{p^\sigma}\\
&=O_q\!\left(\frac{y^{1-\sigma}}{\log{y}}\right)
+\frac1{\phi(q)}\sum_{p>y}p^{-\sigma}.
\end{aligned}
$$
Finally, by the prime number theorem we have
$\sum_{p>y}p^{-\sigma}\gg\frac{y^{1-\sigma}}{(\sigma-1)\log{y}}$,
uniformly for $y\ge\frac32$ and $\sigma\in(1,2]$.
The lemma follows.
\end{proof}

\subsection{A few lemmas}
In the remainder of this section we discuss the topology of $\GL_n(\C)$
and prove some simple lemmas, to be used in the more technical
propositions which follow.

Let $\Mat_{n\times n}(\C)$ denote the set of $n\times n$ matrices with
entries in $\C$. For $A=(a_{ij})\in\Mat_{n\times n}(\C)$,
the {\itshape Frobenius norm} is defined by
\[
\|A\|=\sqrt{\tr\left(\overline{A}^T A\right)}=\sqrt{\sum|a_{ij}|^2}.
\]
Note that this agrees with the Euclidean norm under the identification
of $\Mat_{n\times n}(\C)$ with $\C^{n^2}$.
By the Schwarz inequality, we have
$|Av|\le\|A\|\cdot|v|$ for any $A\in\Mat_{n\times n}(\C)$ and $v\in\C^n$.

We endow $\GL_n(\C)=\{g\in\Mat_{n\times n}(\C):\det{g}\ne0\}$
with the subspace topology. In particular,
it is easy to see that a set $K\subseteq\GL_n(\C)$ is
compact if and only if $K$ is closed in $\Mat_{n\times n}(\C)$
and there are positive real numbers $c$ and $C$ such that
$$
\|g\|\le C\text{ and }|\det{g}|\ge c
\quad\text{for all }g\in K.
$$
Since $g^{-1}$ can be expressed in terms of $\frac1{\det{g}}$ and the
cofactor matrix of $g$, it follows that $\|g^{-1}\|$ is bounded on $K$
(and indeed the map $g\mapsto g^{-1}$ is continuous, so that
$\GL_n(\C)$ is a topological group with this topology).

\begin{lemma}
Suppose $K$ is a compact subset of $\GL_n(\C)$, $g\in K$,
and $U\subseteq\C^n$ contains an open $\delta$-neighborhood of some
point.  Then $gU$ contains an $\varepsilon$-neighborhood,
where $\varepsilon > 0$ depends only on $\delta$ and $K$.
\end{lemma}\label{lem:nbhd_stability}
\begin{proof}
By linearity, we may assume without loss of generality that $U$ contains
the $\delta$-neighborhood of the origin, $N_\delta$. Since $K$ is
compact, there is a number $C>0$ such that $\|g^{-1}\|\le C$ for all
$g\in K$.  Put $\varepsilon=C^{-1}\delta$, and let $N_\varepsilon$ be the
$\varepsilon$-neighborhood of the origin. For any $v\in N_\varepsilon$
we have $|g^{-1}v|\le\|g^{-1}\|\cdot|v|<C\varepsilon=\delta$, so that
$v=g(g^{-1}v)\in gN_\delta$. Since $v$ was arbitrary, $gN_\delta\supseteq
N_\varepsilon$.
\end{proof}

\begin{lemma}\label{lem:sq_root_avg}
For any $v_0, \ldots, v_k \in \C^n$, there exist
$\theta_0, \ldots, \theta_k \in [0, 1]$ such that
\[
\left| \sum_{j = 0}^k e(\theta_k) v_j\right| \leq \sqrt{\sum_{j=0}^k|v_j|^2}.
\]
\end{lemma}
\begin{proof}
We have
\[
\int_{[0, 1]^k}\left|\sum_{j=0}^k e(\theta_j) v_j\right|^2
d\theta_1\cdots d\theta_k = \sum_{j=0}^{k}|v_j|^2.
\]
Thus, the average choice of $(\theta_0,\ldots,\theta_k)$ satisfies the
conclusion.
\end{proof}

\begin{lemma}\label{lem:monomial}
Let $P\in\C[x_1,\ldots,x_n]$. Suppose that every solution to the
equation $P(x_1,\ldots,x_n)=0$ satisfies $x_1\cdots x_n=0$.
Then $P$ is a monomial, i.e., $P=cx_1^{d_1}\ldots x_n^{d_n}$ for 
some $c\neq0$ and non-negative integers $d_1,\ldots,d_n$.
\end{lemma}
\begin{proof}
Let $V=\{(x_1,\ldots,x_n)\in\C^n:P(x_1,\ldots,x_n)=0\}$ be the vanishing
set of $P$. By hypothesis, the polynomial $x_1\cdots x_n$ vanishes on
$V$. Thus, since $\C$ is algebraically closed, Hilbert's Nullstellensatz
implies that there is some $d\in\Z_{\ge0}$ such that $(x_1\cdots x_n)^d$
is contained in the ideal generated by $P$, i.e.\ $P|(x_1\cdots x_n)^d$.
Since $\C[x_1,\ldots,x_n]$ is a unique factorization domain,
this is only possible if $P$ is a monomial.
\end{proof}

\begin{lemma}\label{lem:poly_cont}
Let $P\in\C[x_1,\ldots,x_n]$ and
suppose that $y\in\C^n$ is a zero of $P$. Then for any $\varepsilon>0$ there exists
$\delta>0$ such that any polynomial $Q\in\C[x_1,\ldots,x_n]$, obtained by changing any of the nonzero coefficients
of $P$ by at most $\delta$ each, has a zero $z\in\C^n$ with $|y-z|<\varepsilon$.
\end{lemma}
\begin{proof}
If $P$ is identically $0$ then so is $Q$, so we may take $z=y$.
Otherwise, set
$$
p(t)=P(y+tu)\quad\text{and}\quad q(t)=Q(y+tu)
$$
for $t\in\C$, where $u$ is any unit vector for which 
$p(t)$ does not vanish for all $t$; shrinking $\varepsilon$ if necessary, assume that
$p(t)$ does not vanish on $C_\varepsilon=\{t\in\C:|t|=\varepsilon\}$; and
let $\gamma>0$ be the minimum of $|p(t)|$ on $C_\varepsilon$. For
$t\in C_\varepsilon$ we have
\[
|q(t) - p(t)| < \delta N \Big(1+\varepsilon+|y|\Big)^{\deg P}
\]
where $N$ is the number of nonzero coefficients of $P$. 
Choosing $\delta$ so that the right side of this expression is bounded by $\gamma$, we have
$|q(t)-p(t)|<|p(t)|$ for $t\in C_\varepsilon$.
By Rouch\'e's theorem $q(t)$ has a zero $t_0$ of modulus
$|t_0|<\varepsilon$, and taking $z=y+t_0u$ completes the proof.
\end{proof}

\section{Simultaneous representations of $n$-tuples of complex numbers}

The technical heart of our work is the following analogue of Lemma 2 of \cite{SW}:

\begin{proposition}\label{prop:lem2}
For any real numbers $y,R>1$ there exists $\eta>0$ such that,
for all $\sigma\in(1,1+\eta]$, we have
\begin{align*}
\bigg\{\biggl(\prod_{p > y}
L(\sigma+it_p,\pi_{j,p})\biggr)_{j=1,\ldots,n}&\ :
t_p\in\R\text{ for each prime }p>y\bigg\}\\
&\supseteq
\Big\{ (z_1, \ldots, z_n) \in \C^n :
R^{-1} \leq |z_j| \leq R \textnormal{ for all } j \Big\}.
\end{align*}
\end{proposition}
Loosely speaking, after simultaneously approximating the $t_p$ by a common
$t$, it will follow that we can make the $L(s, \pi_j)$ independently
approach any desired $n$-tuple of nonzero complex numbers, and this will
allow us to find zeros in linear or polynomial combinations.

The proof relies on an analogue of Lemma 1 of \cite{SW}, whose adapation
is not especially straightforward. We carry out this work by proving
two technical propositions; the first establishes the existence of
solutions to a certain equation involving matrices in a fixed compact
subset of $\GL_n(\C)$.

\begin{proposition}\label{prop:td}
Let
\[
T = \{ (z_1, \ldots, z_n) \in \C^n \ : \ |z_1| = \ldots = |z_n| = 1 \},
\]
\[
D = \{ (z_1, \ldots, z_n) \in \C^n \ : \ |z_1|, \ldots, |z_n| \leq 1 \},
\]
and fix a compact set $K \subseteq \GL_n(\C)$.
Then there is a number $m_0 > 0$ such that for every $m \geq m_0$ and all
$(g_1, \ldots, g_m) \in K^m$, there are continuous functions
$f_1, \ldots, f_m : D \rightarrow T$ such that
$\sum_{i = 1}^m g_i f_i(z) = z$ for all $z \in D$.
\end{proposition}

We will carry out the proof in three steps:
\begin{enumerate}[(1)]
\item\label{it_small_neigh}
We first show that there exist $\varepsilon > 0$ and
$m_1$ such that
for all $m \geq m_1$ and all $(g_1, \ldots, g_m) \in K^m$, the set
$\{ \sum_{i = 1}^m g_i t_i : t_1,\ldots,t_m \in T \}$
contains an open $\varepsilon$-neighborhood of a point in $\C^n$.
\item\label{it_big_neigh}
`Fattening' the neighborhood constructed in the first step, we will 
show that there exists $m_2$ such that for $m \geq m_2$ and all
$(g_1,\ldots,g_m)\in K^m$,
$\{ \sum_{i = 1}^m g_i t_i : t_1,\ldots,t_m \in T \}$
contains the closed ball of radius $2$,
$\{(z_1,\ldots,z_n):|z_1|^2+\ldots+|z_n|^2\le4\}$.
\item\label{it_cont}
Although the previous step yields a parametrization of a large closed
set, it is not
obviously continuous. By repeating the construction from step
\eqref{it_small_neigh} using the added knowledge of step \eqref{it_big_neigh},
we show that one can achieve a continuous parametrization of $D$.
\end{enumerate}

\begin{proof}
We begin by showing \eqref{it_small_neigh}. By compactness, 
there is an $m_1$ such that for any $m\ge m_1$ and any $m$-tuple
$(g_1, \ldots, g_m)$, there is a distinct pair of indices $i,j$
such that $\| g_i^{-1} g_j - I \| < \frac{1}{3 \sqrt{n}}$.
Assume, without loss of generality, that $(i,j)=(1,2)$, and
put $\Delta = g_1^{-1} g_2 - I$. Then for any choice of $t_1, t_2 \in T$,
we have
\[
g_1 t_1 + g_2 t_2 = g_1 (t_1 + (I + \Delta) t_2),
\]
where $\| \Delta \|<\frac1{3\sqrt{n}}$.

We introduce some notation. First, define $A=\{z\in\C:|z-1|\le\frac13\}$
and $B=\{ z \in \C : |z - 1| \leq \frac{2}{3} \}$.
Next, let $s_1, s_2: B \rightarrow \C$ be the unique continuous
functions satisfying $z = s_1(z) + s_2(z)$,
$|s_1(z)| = |s_2(z)| = 1$ and $\Im(\frac{s_1(z)}{s_2(z)}) > 0$ for all
$z\in B$. For $j=1,2$, let $t_j:B^n\rightarrow T$ be defined by
$t_j(z_1,\ldots,z_n)=(s_j(z_1),\ldots,s_j(z_n))$.

Given an arbitrary element $w\in A^n$,
we define a continuous function $h_w: B^n \rightarrow \C^n$ by
$h_w(z) = w - \Delta{t_2(z)}$.
Since $|t_2(z)|=\sqrt{n}$ and
$\|\Delta\|<\frac1{3\sqrt{n}}$, we have
$|\Delta t_2(z)|<\frac13$.
In particular, each entry of $\Delta t_2(z)$ is bounded in magnitude
by $\frac13$, so by the triangle inequality, the
image of $h_w$ is contained in $B^n$. By the Brouwer fixed point theorem,
there exists $z \in B^n$ with $h_w(z) = z$, so that
\[
t_1(z) + (I + \Delta) t_2(z) = z+\Delta{t_2(z)} = z+w-h_w(z) = w.
\]
Therefore, all of $A^n$ is in the image of the map
$z \mapsto t_1(z) + (I + \Delta) t_2(z)$,
so that in particular
\[
A^n \subseteq \{ t_1 + g_1^{-1} g_2 t_2 \ : \ t_1, t_2 \in T \}.
\]

Applying Lemma \ref{lem:nbhd_stability} with $\delta=\frac13$,
there is an $\varepsilon>0$
depending only on $K$ such that
$\{ g_1 t_1 + g_2 t_2 \ : \ t_1, t_2 \in T \}$ 
contains an $\varepsilon$-neighborhood of some point in $\C^n$.
We conclude the same of the set
$\{ g_1 t_1 + \ldots + g_m t_m \ : \ t_1, \ldots, t_m \in T \}$
by choosing arbitrary fixed $t_3, \ldots, t_m \in T$.
\\
\\
Proceeding to step \eqref{it_big_neigh}, let $k_1$ be a large integer
to be determined later, set $m_2 = m_1 k_1$, and for any $m \geq m_2$ write
$m = k m_1 + l$ 
with $k\ge k_1$ and
$0 \leq l < m_1$.

For each $j$ with $0\leq j<k$, applying step \eqref{it_small_neigh}
to $(g_{jm_1+1},\ldots,g_{jm_1+m_1})$, we obtain an $\varepsilon$-neighborhood
centered at some $v_j\in \C^n$. Further, we put
$v_k = g_{km_1 + 1} \overrightarrow{1} + \ldots
+ g_{k m_1 + l} \overrightarrow{1}$, where
$\overrightarrow{1} = (1, \ldots, 1) \in T$.
Since $m_1$ is fixed and $K$ is compact, 
we have $|v_j| \leq C$ for $0\leq j\leq k$, for some $C$ independent of the individual $g_i$.

Let
$N_\varepsilon=\{(z_1,\ldots,z_n)\in\C^n:|z_1|^2+\ldots+|z_n|^2<\varepsilon^2\}$
be the $\varepsilon$-neighborhood of the origin in $\C^n$. Then by the
above observations, for any $\theta_0,\ldots,\theta_k\in[0,1]$,
$\big\{ \sum_{i = 1}^m g_i t_i \, : \, t_1,\ldots,t_m \in T \big\}$
contains the set
$$
\sum_{j=0}^{k-1}e( \theta_j)(v_j+N_{\varepsilon})+e(\theta_k)v_k
=\sum_{j=0}^ke(\theta_j)v_j+kN_{\varepsilon}.
$$
By Lemma~\ref{lem:sq_root_avg}, there is a choice of
$\theta_0,\ldots,\theta_k$ for which
$\bigl|\sum_{j=0}^ke(\theta_j)v_j\bigr|\le C\sqrt{k+1}$.
Now let $k_1$ be the smallest positive integer satisfying
$k_1\varepsilon > C \sqrt{k_1 + 1} + 2$. Then for $k\ge k_1$, we have
shown that $\{ \sum_{i = 1}^m g_i t_i : t_1,\ldots,t_m\in T \}$
contains the closed ball of radius $2$.
\\
\\
Proceeding to step \eqref{it_cont}, we put $m_0=3nm_2$. Suppose that 
$m\geq m_0$ and $(g_1,\dots,g_m)$ are given, and choose a partition
of $\{1,\ldots,m\}$ into $3n$
sets $I_{j,\ell}$ (for $1\leq j\leq n$, $1\leq \ell\leq3$),
each of size at least $m_2$.
For each $j$ with $1\leq j\leq n$, write
\[
v_j = v_{j,1} = v_{j,2} = v_{j,3} = (0,\ldots,0,2,0,\ldots,0),
\]
where the $2$ is in the $j$th position. For each $j$ and $\ell$ we use
step \eqref{it_big_neigh} to express $v_{j,\ell}$ in the form
\begin{equation}\label{eqn:step3_idecomp}
v_{j,\ell} = \sum_{i\in I_{j,\ell}}g_i t_i
\end{equation}
for some $t_i\in T$.

Next, note that the set
$\bigl\{2[(1,\ldots,1)+\alpha+\beta]\ : \alpha,\beta\in T\}$
contains $D$. As in the proof of step \eqref{it_small_neigh},
we can choose continuous functions
$\alpha=(\alpha_1,\ldots,\alpha_n), \beta=(\beta_1,\ldots,\beta_n):D\to T$
such that
$z_j=2[1+\alpha_j(z)+\beta_j(z)]$ for every $z=(z_1,\ldots,z_n)\in D$.
Thus,
$$
z=\sum_{j=1}^n[1+\alpha_j(z)+\beta_j(z)]v_j
=\sum_{j=1}^n [v_{j,1}+\alpha_j(z)v_{j,2}+\beta_j(z)v_{j,3}].
$$
Finally, we use \eqref{eqn:step3_idecomp} to rewrite this as
$$
z=\sum_{j=1}^n\left[
\sum_{i\in I_{j,1}}g_it_i 
+\sum_{i\in I_{j,2}}g_i\big(t_i\alpha_j(z)\big)
+\sum_{i\in I_{j,3}}g_i\big(t_i\beta_j(z)\big)\right],
$$
which is a decomposition of the type required.
\end{proof}

Next, we use the quasi-orthogonality of the coefficients $\lambda_j(p)$
(Lemma~\ref{lem:sump}) to show that, by choosing an arbitrary ``twist''
$\epsilon_p\in S^1$ for each large prime $p$, we can make sums of the
$\epsilon_p\lambda_j(p)$ line up in linearly independent directions,
as quantified in the following proposition.

Given a real parameter $y>0$, we write
\[
S(y) = \{p \text{ prime}: p > y \}
\quad\text{and}\quad
s(y,\sigma) = \sum_{p\in S(y)} p^{-\sigma}.
\]

\begin{proposition}\label{prop:sw1}
There is a compact set $K \subseteq \GL_n(\C)$, explicitly defined in \eqref{eqn:def_K}
depending only on the
degrees $r_1,\ldots,r_n$, with the following property:

Let $m$ be a positive integer. Then there is a real number
$\delta>0$ (depending on the $\pi_j$ and $m$) such that for any $y>0$
and any $\sigma\in(1,1+\delta]$,
there exists a partition of $S(y)$ into $mn$ pairwise disjoint subsets
$S_{ik}(y)$ ($i=1,\ldots,m$, $k=1,\ldots,n$)
and a choice of $\epsilon_p \in S^1$ for each $p\in S(y)$, such that the
$m$-tuple of matrices $(g_1,\ldots,g_m)$ defined by
\begin{equation}\label{eqn:gidef}
g_i = \bigg( \frac{mn}{s(y,\sigma)} \sum_{p \in S_{ik}(y)}
\frac{ \epsilon_p \lambda_j(p)}{p^{\sigma}} \bigg)_{1 \leq j, k \leq n},
\quad i=1,\ldots,m
\end{equation}
lies in $K^m$.
\end{proposition}

\begin{proof}
Let $q$ be the smallest prime number satisfying $q\equiv1\pmod{mn}$ and
$q\nmid\prod_{j=1}^n\cond(\pi_j)$. We put $t=\frac{q-1}{mn}$ and
define $S_{ik}^\circ(y)$ to be the union of residue classes
$$
S_{ik}^\circ(y)=\bigcup_{\ell=1}^t
\bigl\{p\in S(y):p\equiv tn(i-1)+t(k-1)+\ell\pmod{q}\bigr\},
$$
and
$$
S_{ik}(y)=\begin{cases}
S_{ik}^\circ(y)\cup\{q\}&\text{if }
i=k=1\text{ and }y<q,\\
S_{ik}^\circ(y)&\text{otherwise}.
\end{cases}
$$
Then the $S_{ik}(y)$ are pairwise disjoint and cover $S(y)$.

For a fixed choice of $i$, let $v_k$ denote the $k$th column of $g_i$,
as defined in \eqref{eqn:gidef}, with the $\epsilon_p$ yet to be chosen.
We will show by induction that there is a choice of the $\epsilon_p$ such that
\begin{equation}\label{eqn:projspan}
|v_\ell - \proj_{\textspan\{v_1, \ldots, v_{\ell - 1}\}} v_\ell| \geq
\frac{1}{2r}
\end{equation}
holds for every $\ell=1,\ldots,n$, where $r=\sqrt{r_1^2+\ldots+r_n^2}$.
To that end,
let $k$ be given, and assume that \eqref{eqn:projspan} has been
established for $\ell=1,\ldots,k-1$.  Choose a unit vector
$u = (u_1, \ldots, u_n)$ orthogonal to $v_1, \ldots, v_{k - 1}$.
By the Schwarz inequality and the Ramanujan bound
$|\lambda_j(p)|\le r_j$, for each prime $p$ we have
$
|\bar{u}_1 \lambda_1(p) + \ldots + \bar{u}_n \lambda_n(p)| \leq
r.
$
Therefore
\begin{align}\label{eqn:inn_prod}
\frac{mn}{s(y, \sigma)}
\sum_{p\in S_{ik}(y)}
\frac{|\bar{u}_1\lambda_1(p)+\ldots+\bar{u}_n\lambda_n(p)|}{p^\sigma}
&\geq\frac{mn}{r s(y, \sigma)} \sum_{p\in S_{ik}^\circ(y)}
\frac{|\bar{u}_1\lambda_1(p)+\ldots+\bar{u}_n\lambda_n(p)|^2}{p^\sigma}\\
&=\frac{1+O_{m, n}(\sigma-1)}{r}, \nonumber
\end{align}
the latter equality following by Lemma \ref{lem:sump}.
We choose $\delta$ so that the $O$ term
above is bounded in modulus by $\frac12$, and
for each $p \in S_{ik}(y)$ we choose $\epsilon_p$ such that
$\epsilon_p (\bar{u}_1 \lambda_1(p) + \ldots + \bar{u}_n \lambda_n(p))$
is real and nonnegative. Then 
the left side of \eqref{eqn:inn_prod} equals
\[
\langle u, v_k \rangle = 
\langle u, v_k - \proj_{\textspan\{v_1, \ldots, v_{k - 1}\}} v_k \rangle \leq
|v_k - \proj_{\textspan\{v_1, \ldots, v_{k - 1}\}} v_k|,
\]
so that \eqref{eqn:projspan} follows for $\ell=k$.

Applying Gram--Schmidt orthogonalization to $v_1,\ldots,v_n$, it follows
from \eqref{eqn:projspan} that
$|\det{g_i}|\ge (2r)^{-n}$. Moreover, by the Schwarz inequality
and Lemma \ref{lem:sump} again, each entry of $g_i$ is bounded above by
$1+O_{m, n}(\sigma-1)$,
so that $\|g_i\|\le2n$ for a suitable choice of $\delta$. Thus,
\begin{equation}\label{eqn:def_K}
K=\{g\in\GL_n(\C):\|g\|\le 2n, |\det{g}|\ge(2r)^{-n}\}
\end{equation}
has the desired properties.
\end{proof}

We are now ready to prove Proposition \ref{prop:lem2}, largely following \cite{SW}.

\begin{proof}[Proof of Proposition \ref{prop:lem2}]
We use Propositions \ref{prop:sw1} and \ref{prop:td} to 
determine a compact set $K\subseteq\GL_n(\C)$, a positive integer
$m_0$, and a real number $\delta>0$ with the properties described
there. 
Taking $m=m_0$,
the aforementioned propositions yield, for any
$\sigma\in(1,1+\delta]$, an $m$-tuple of matrices
$(g_1,\ldots,g_m)\in K^m$, elements $\epsilon_p \in S^1$ for each prime $p > y$, and
continuous functions $f_1,\dots,f_m:D\rightarrow T$
such that
\begin{equation}\label{eqn:id_decomp}
\sum_{i=1}^m g_i f_i(z) = z\quad\text{for all }z\in D.
\end{equation}

Now, let $\mu=\frac{s(y,\sigma)}{mn}$.
For each prime $p>y$, we define
a continuous function $t_p : \mu D\rightarrow\R$ satisfying
\begin{equation}\label{eqn:def_tp}
p^{- i t_p(z)} = \epsilon_p f_i(\mu^{-1}z)_k,
\end{equation}
where $(i,k)$ is the unique pair of indices for which $p\in S_{ik}(y)$ and
$f_i(\mu^{-1}z)_k$ denotes the $k$th component of
$f_i(\mu^{-1}z)$.
(Note that the lift from $S^1$ to $\R$ is
possible since $D$ is simply connected.)

Define an error term $E(z)=(E_1(z),\ldots,E_n(z))$ by writing,
for each $j=1,\ldots,n$, 
$$
E_j(z)=\sum_{p>y}\left(\log L(\sigma+it_p(z),\pi_{j,p}) 
-\lambda_j(p)p^{-(\sigma+it_p(z))}\right).
$$
By the Ramanujan bound, we have
$$
\log L(s,\pi_{j,p})-\lambda_j(p)p^{-s}=O(p^{-2})
$$
uniformly for $\Re(s)\ge1$. Since $\sum_p p^{-2}$ converges,
the continuity of $E$ follows from that of the individual $t_p$.
Moreover, each component $E_j(z)$ is bounded by a number $C>0$, independent
of $j$, $z$, $y$, or $\sigma$. 

Set $R'=\sqrt{\pi^2+\log^2{R}}$.  We take $\eta\in(0,\delta]$
small enough that the condition $\sigma\in(1,1+\eta]$ ensures that
$\mu\ge C+R'$.
By \eqref{eqn:id_decomp}, \eqref{eqn:def_tp}, and
Proposition \ref{prop:sw1} we have
$$
\sum_{p>y}\lambda_j(p)p^{-(\sigma+it_p(z))}=
\sum_{i=1}^m\sum_{k=1}^n
\sum_{p\in S_{ik}(y)} 
\frac{\lambda_j(p)\epsilon_pf_i(\mu^{-1}z)_k}{p^{\sigma}}=z_j,
$$
for any $z=(z_1,\ldots,z_n)\in\mu{D}$.
Now fix $w\in R'D$ and define a function
$F_w:(C+R')D\to\C$ by $F_w(z)=w-E(z)$.
By the estimate for $E_j(z)$ above, the image of $F_w$ is contained
in $(C+R')D$. Thus,
by the Brouwer fixed point theorem, there exists $z\in(C+R')D$
with $F_w(z)=z$, so that
\[
\biggl(\sum_{p>y}\log L(\sigma+it_p(z),\pi_{j,p})\biggr)_{j=1,\ldots,n}
=z+E(z)=z+w-F_w(z)=w.
\]
Taking exponentials yields the proposition.
\end{proof}

\section{Proof of Theorem \ref{thm:main}}\label{sec:main}
The proof will be carried out in two steps:
\begin{enumerate}[(1)]
\item
Applying our previous results, we
show that unless $P$ is a monomial (as described in Theorem \ref{thm:main}), 
for every $\sigma > 1$ sufficiently close to $1$
there are real numbers $t_p$ (for each prime $p$) and $t_0$ 
such that $P|_{s=\sigma+it_0}$
vanishes at $\bigl(\prod_pL(\sigma+it_p,\pi_{1,p}),\ldots,
\prod_pL(\sigma+it_p,\pi_{n,p})\bigr)$.
\item Simultaneously approximating the $p^{-it_p}$ by $p^{-it}$ for a
common value of $t$, we use Rouch\'e's theorem to find a zero of
$P(L(s,\pi_1),\ldots,L(s,\pi_n))$ close to $\sigma+it$.
\end{enumerate}
Note that the second step is standard and is applied in \cite{SW}
in much the same way.
\\
\\
We begin with a polynomial $P$ whose coefficients are 
finite Dirichlet series $D(s)=\sum_{m=1}^M a_mm^{-s}$, and let
$y$ be the largest value of $M$ occurring in any of these coefficients.
We rewrite each $L(s,\pi_j)$ as
$L_{\leq{y}}(s,\pi_j)L_{>y}(s,\pi_j)$,
splitting each Euler product into products over primes $p\leq y$
and $p>y$ respectively. Setting
$$
Q(x_1,\ldots,x_n)
=P\bigl(L_{\le{y}}(s,\pi_1)x_1,\ldots,L_{\le{y}}(s,\pi_n)x_n\bigr),
$$
we have
$P(L(s,\pi_1),\dots,L(s,\pi_n))=Q(L_{>y}(s,\pi_1),\dots,L_{>y}(s,\pi_n))$.

The coefficients of $Q$ are rational functions of
the $p^{-s}$ for $p\le y$.  More precisely, for any monomial term
$D(s)x_1^{d_1}\cdots x_n^{d_n}$ in the expansion of $P$, the corresponding
term of $Q$ is 
\[D(s)L_{\le{y}}(s,\pi_1)^{d_1}\cdots{}L_{\le{y}}(s,\pi_n)^{d_n}
x_1^{d_1}\cdots{}x_n^{d_n}.
\]  Since the finite
Euler products $L_{\le{y}}(s,\pi_j)$ are non-vanishing holomorphic functions
on $\{s\in\C:\Re(s)\ge1\}$, the corresponding terms of $P$ and $Q$
have the same zeros there.

Let $D_1(s),\ldots,D_m(s)$ run through the coefficients of $P$ which do
not vanish identically, and consider their product
$f(s)=D_1(s)\cdots D_m(s)$. Then $f$ is itself a finite Dirichlet series
which does not vanish identically. By complex analysis, $f$ cannot
vanish at $1+it$ for every $t\in\R$, so there is some $t_0$
for which $D_1(1+it_0),\ldots,D_m(1+it_0)$ are all non-zero, and the same holds for the corresponding terms of $Q$.

Next we 
specialize the coefficients of $Q$ to a fixed value of $s$, obtaining
a polynomial $h_s\in\C[x_1,\dots,x_n]$. Considering $s=1+it_0$, Lemma \ref{lem:monomial}
implies that either $h_{1+it_0}=cx_1^{d_1}\cdots x_n^{d_n}$
for some $c\in\C$
and $d_1,\ldots,d_n\in\Z_{\ge0}$, or that there are $y_1,\ldots,y_n\in\C$,
none zero, for which $h_{1+it_0}(y_1,\ldots,y_n)=0$.
In the former case, it follows from our choice of $t_0$ that
$P=D(s)x_1^{d_1}\cdots x_n^{d_n}$ is a monomial, as allowed in the
conclusion of Theorem~\ref{thm:main}.
Henceforth we assume that we are in the latter case,
and aim to show that $P(L(s,\pi_1),\ldots,L(s,\pi_n))$
has a zero with $\Re(s)>1$.

We choose $R>1$ so that $R^{-{1/2}}\leq|y_j|\leq R^{1/2}$ for
every $j$.  By Lemma \ref{lem:poly_cont}, there is a number
$\varepsilon>0$ such that for every $\sigma\in(1,1+\varepsilon]$,
there exists $(z_1(\sigma),\ldots,z_n(\sigma))\in\C^n$
satisfying $h_{\sigma+it_0}(z_1(\sigma),\ldots,z_n(\sigma))=0$
and $R^{-1}\leq|z_j(\sigma)|\leq R$ for every $j$.  We use
Proposition \ref{prop:lem2} to determine $\eta$ in terms of
$y$ and $R$, and assume that $\eta \le \varepsilon$
by shrinking $\eta$ if necessary.
Proposition~\ref{prop:lem2} then guarantees that, for every
$\sigma\in(1,1+\eta]$, we can solve the
simultaneous system of equations
$$
\prod_{p>y}L(\sigma+it_p,\pi_{j,p})=z_j(\sigma),\quad j=1,\ldots,n,
$$
in the $t_p$ for $p>y$. For $p\le y$ we set $t_p=t_0$, thereby
completing step (1).
\\
\\
Turning to step (2), let $\sigma_1,\sigma_2\in\R$ with
$1<\sigma_1<\sigma_2\le1+\eta$, and put $\sigma=\frac{\sigma_1+\sigma_2}2$.
With the $t_0$ and $t_p$ resulting from step (1) for this choice of $\sigma$,
let $P_{it_0}$ denote the polynomial obtained from $P$ by replacing $s$
by $s+it_0$ in all of its coefficients, and define
\begin{equation}\label{eqn:Fdef}
F(s)=P_{it_0}\!\left(
\prod_pL(s+it_p,\pi_{1,p}),\ldots,\prod_pL(s+it_p,\pi_{n,p})\right).
\end{equation}
Then $F$ is holomorphic for $|s-\sigma|<\sigma-1$ and satisfies
$F(\sigma)=0$ by construction. It follows that there
is a number $\rho\in(0,\frac{\sigma_2-\sigma_1}2]$ such that
$F(s)\neq0$ for all $s\in C_\rho=\{s\in\C:|s-\sigma|=\rho\}$.
Write $\gamma$ for the minimum of $|F(s)|$ on $C_\rho$.

Next, by abuse of notation, we write $P(s)$ as shorthand for
$P(L(s,\pi_1),\ldots,L(s,\pi_n))$. As
$P(s)=\sum_{m=1}^\infty a_mm^{-s}$ converges absolutely as a Dirichlet
series for $\Re(s) > 1$, there is an integer
$M > 0$ with
$\sum_{m=M}^{\infty}|a_m|m^{-\sigma_1}\le\frac{\gamma}3$.
By \eqref{eqn:Fdef} we have
$F(s)=\sum_{m=1}^\infty b_mm^{-s}$, where
$b_m=a_m\prod_{p|m}p^{-it_p\ord_p(m)}$, and
by the joint uniform distribution of $p^{it}$ for primes $p<M$, it
follows that the set of $t\in\R$ satisfying
\[\sum_{m=1}^{M-1}\frac{|a_mm^{-it}-b_m|}{m^{\sigma_1}}<\frac{\gamma}3\]
has positive lower density.  For any such $t$ the triangle
inequality yields $|P(s+it)-F(s)|<\gamma$ for all $s$ with
$\Re(s)\ge\sigma_1$, and in particular for all $s\in C_\rho$.
By Rouch\'e's theorem, it follows that $P(s+it)$
has a zero $s$ with $|s-\sigma|<\rho$.
Thus, $P(s)$ has zeros with real part in $[\sigma_1,\sigma_2]$, and indeed we have
$$
\#\{s\in\C:\Re(s)\in[\sigma_1,\sigma_2],
\Im(s)\in[-T,T], P(s)=0\}
\gg_{\sigma_1,\sigma_2}T
$$
for all $T\ge T_0(\sigma_1,\sigma_2)$.

\bibliographystyle{amsplain}
\providecommand{\bysame}{\leavevmode\hbox to3em{\hrulefill}\thinspace}
\providecommand{\MR}{\relax\ifhmode\unskip\space\fi MR }
\providecommand{\MRhref}[2]{%
  \href{http://www.ams.org/mathscinet-getitem?mr=#1}{#2}
}
\providecommand{\href}[2]{#2}

\end{document}